\documentclass[a4paper, 12pt]{amsart}
\usepackage[utf8]{inputenc}
\usepackage[english]{babel}
\usepackage[T1]{fontenc}
\usepackage{amsmath, amssymb, amsfonts}
\usepackage{enumerate}
\usepackage{graphicx}
\usepackage{comment}
\usepackage[hmargin=2.25cm, vmargin=2.5cm]{geometry}
\usepackage{tikz}
\usepackage[colorlinks=true, citecolor=blue]{hyperref}

\title{Quantum reflections, random walks and cut-off}
\author{Amaury Freslon}
\address{A. Freslon, Laboratoire de Math\'ematiques d'Orsay, Univ. Paris-Sud, CNRS, Universit\'e Paris-Saclay, 91405 Orsay, France}
\email{amaury.freslon@math.u-psud.fr}
\date{}

\theoremstyle{plain}
\newtheorem{thm}{Theorem}[section]
\newtheorem{prop}[thm]{Proposition}

\newtheorem{lem}[thm]{Lemma}

\theoremstyle{definition}
\newtheorem{de}[thm]{Definition}

\theoremstyle{remark}

\DeclareMathOperator{\Arg}{Arg}

\DeclareMathOperator{\ev}{ev}
\DeclareMathOperator{\id}{id}

\DeclareMathOperator{\Irr}{Irr}
\DeclareMathOperator{\Tr}{Tr}
\DeclareMathOperator{\var}{var}

\newcommand{\C}{\mathbb{C}}

\newcommand{\D}{\Delta}

\newcommand{\G}{\mathbb{G}}

\newcommand{\N}{\mathbb{N}}
\newcommand{\R}{\mathbb{R}}
\newcommand{\T}{\mathbb{T}}
\newcommand{\Z}{\mathbb{Z}}

\newcommand{\dd}{\mathrm{d}}

\renewcommand{\O}{\mathcal{O}}

\begin{document}

\begin{abstract}
We study the cut-off phenomenon for random walks on free unitary quantum groups coming from quantum conjugacy classes of classical reflections. We obtain in particular a quantum analogue of the result of U. Porod concerning certain mixtures of reflections. We also study random walks on quantum reflection groups and more generally on free wreath products of finite groups by quantum permutation groups.
\end{abstract}

\maketitle

\section{Introduction}

Recently, J.P. McCarthy initiated in \cite{maccarthy2017random} the study of the cut-off phenomenon for random walks on compact quantum groups, focusing on the finite case. This led us to adress examples in the compact (infinite) case in \cite{freslon2017cutoff}, where we were able to prove a cut-off for several examples inspired by classical random walks on orthogonal and symmetric groups. These examples include a quantum analogue of the uniform plane Kac walk on free orthogonal quantum groups as well as the uniform random walk on the set of quantum transpositions in free symmetric quantum groups. In the present paper, we continue this work by considering other compact quantum groups in connection with complex reflections.

Our first family of examples is free unitary quantum groups. In the classical case, the unitary group can be thought of as a complex version of the orthogonal group. The same analogy works in the quantum setting in a precise sense : the free unitary quantum group is the \emph{free complexification} of the free orthogonal quantum group. This gives us a convenient way of producing central states (i.e.~conjugation-invariant measures) on this quantum group using the free product construction. We analyse the behaviour of the associated random walk when the state is the uniform measure on the conjugacy class of a classical complex reflection. We then consider a "mixture" of these measures which was studied in the classical case by U. Porod in \cite{porod1996cutII} and prove a cut-off with the same parameter, namely $N\ln(N)/2$.

In \cite{freslon2017cutoff}, we were not able to prove a cut-off for arbitrary mixtures of rotations because could not handle the behaviour of our approximations when the angle was close to $0$. We therefore had to assume that the measure governing the mixture had a support bounded away from $0$. Here, since we consider a specific measure, we can do explicit computations even though the support is not bounded away from $0$ and directly prove the cut-off without reducing the problem to individual conjugacy classes of reflections.

Our second family of examples comes from reflection groups. More precisely, J. Bichon introduced in \cite{bichon2004free} a quantum version of the wreath product construction called the \emph{free wreath product} of the dual of a discrete group $\Gamma$ by the free symmetric quantum group $S_{N}^{+}$. When $\Gamma = \Z_{s}$ is cyclic, this object coincides with a quantum analogue of the complex reflection group $H_{N}^{s}$ introduced by T. Banica and R. Vergnioux in \cite{banica2009fusion}. It turns out that assuming $\Gamma$ to be cyclic does not simplify the arguments of the proof of the cut-off phenomenon, so that we will consider in full generality the free wreath product of the dual of a finite group by $S_{N}^{+}$. Two problems arise in this context. First, constructing states on these compact quantum groups is not straightforward. It can nevertheless be done using a quantum group tool called monoidal equivalence and some earlier results on approximation properties for quantum $SU(2)$ groups. Second, the computations are more involved than in the previous examples and we have to make an assumption on the state considered (more precisely on the parameter $\tau$ defining the state) so that we do not get a cut-off statement for all pure central states.

Let us conclude by an outline of the paper. Section \ref{sec:preliminaries} contains some notations and essential results from \cite{freslon2017cutoff}. In Section \ref{sec:unitary} we introduce free unitary quantum groups and prove several cut-off statements. In particular, we obtain in Theorem \ref{thm:cutoffunitary} a cut-off at $N\ln(N)/2$ for a particular mixture of random reflections in perfect analogy with the classical result. In Section \ref{sec:reflection} we turn to free wreath products and prove in Theorem \ref{thm:cutoffwreath} a cut-off for pure central states satisfying a simple condition on the parameter $\tau$ defining the state.

\section{Preliminaries}\label{sec:preliminaries}

We first briefly recall some facts concerning random walks on compact quantum groups, mainly to fix notations. The reader may refer to \cite[Sec 2]{freslon2017cutoff} for details and proofs. A \emph{compact quantum group} $\G$ is given by a Hopf $*$-algebra $\O(\G)$ together with an invariant state $h$, i.e.~a positive linear map $\O(\G)\to\C$ satisfying $h(1) = 1$ and
\begin{equation*}
(h\otimes\id)\circ \D(x) = h(x).1 = (\id\otimes h)\circ\D(x)
\end{equation*}
for all $x\in \O(\G)$, where $\D : \O(\G)\to \O(\G)\otimes\O(\G)$ is the coproduct. The state $h$ is called the \emph{Haar state} of $\G$. A \emph{unitary representation of dimension $n$} of $\G$ is a unitary element $u\in M_{n}(\O(\G))$ such that for all $1\leqslant i, j\leqslant n$,
\begin{equation*}
\D(u_{ij}) = \sum_{k=1}^{n}u_{ik}\otimes u_{kj}.
\end{equation*}
A \emph{morphism} between representations $u$ and $v$ of dimension $n$ and $m$ respectively is a linear map $T : \C^{n}\rightarrow \C^{m}$ such that $(T\otimes \id)u = v(T\otimes \id)$. Two representations are said to be \emph{equivalent} if there is a bijective morphism between them. A representation $u$ is said to be \emph{irreducible} if the only morphisms between $u$ and itself are the scalar multiples of the identity. A quantum analogue of the Peter-Weyl theory ensures that any unitary representation splits as a direct sum of irreducible finite-dimensional ones, so that the latter are enough to describe the representation theory. Denoting by $\Irr(\G)$ the set of equivalence classes of irreducible representations, we define for $\alpha\in \Irr(\G)$ the corresponding \emph{character} $\chi_{\alpha}$ as
\begin{equation*}
\chi_{\alpha} = \Tr(u^{\alpha}) = \sum_{i=1}^{\dim(u^{\alpha})}u^{\alpha}_{ii}
\end{equation*}
for any representative $u^{\alpha}$ of $\alpha$.

In the present work, we are interested in the sequence $(\varphi^{\ast k})_{k\in \N}$ of convolution powers of a state $\varphi$ on the Hopf $*$-algebra $\O(\G)$ of a compact quantum group $\G$, which is recursively defined by $\varphi^{\ast 1} = \varphi$ and
\begin{equation*}
\varphi^{\ast(k+1)} = (\varphi^{\ast k}\otimes\varphi)\circ\D = (\varphi\otimes\varphi^{\ast k})\circ\D.
\end{equation*}
We will estimate how far states are from the Haar state with respect to the \emph{total variation distance}. To define it requires the introduction of some topology. The state $h$ induces a scalar product on $\O(\G)$ through the formula $\langle x, y\rangle_{h} = h(xy^{*})$ which, after separation and completion, yields a Hilbert space $L^{2}(\G)$. The action by left multiplication induces an embedding $\O(\G)\hookrightarrow B(L^{2}(\G))$ yielding a norm  on $\O(\G)$ denoted by $\|\cdot\|_{\infty}$. The weak closure of $\O(\G)$ is then a von Neumann algebra denoted by $L^{\infty}(\G)$.

\begin{de}\label{de:totalvariation}
The total variation distance between two states $\varphi$ and $\psi$ on the von Neumann algebra $L^{\infty}(\G)$ of $\G$ is defined by
\begin{equation*}
\|\varphi - \psi\|_{TV} = \sup_{p\in \mathcal{P}(L^{\infty}(\G))}\vert \varphi(p) - \psi(p)\vert,
\end{equation*}
where $\mathcal{P}(L^{\infty}(\G)) = \{p\in L^{\infty}(\G) \mid p^{2} = p = p^{*}\}$.
\end{de}

The main tool to study this distance is the \emph{Upper Bound Lemma} of P. Diaconis and M. Shahshahani \cite{diaconis1981generating}. We will simply state its quantum version here and refer the reader to \cite[Lem 2.7]{freslon2017cutoff} for details and proofs. For a state $\varphi$ on $\O(\G)$ and an irreducible representation $\alpha\in \Irr(\G)$ of dimension $d_{\alpha}$, we denote by $\widehat{\varphi}(\alpha)$ the matrix $(\varphi(u^{\alpha}_{ij}))_{1\leqslant i, j\leqslant d_{\alpha}}$, where $(u^{\alpha}_{ij})_{1\leqslant i, j\leqslant d_{\alpha}}$ is a representative of $\alpha$.

\begin{lem}[Upper bound lemma]\label{lem:upperbound}
Let $\G$ be a compact quantum group of Kac type and let $\varphi$ be a state on $\G$ which extends to $L^{\infty}(\G)$. Then,
\begin{equation*}
\|\varphi^{\ast k} - h\|_{TV}^{2} \leqslant \frac{1}{4}\sum_{\alpha\in \Irr(\G)\setminus\{1_{\G}\}}d_{\alpha}\Tr\left(\widehat{\varphi}(\alpha)^{* k}\widehat{\varphi}(\alpha)^{k}\right),
\end{equation*}
where $1_{\G} = 1\in M_{1}(\O(\G))$ denotes the trivial representation.
\end{lem}

As in \cite{freslon2017cutoff} we will focus on \emph{central states}, for which the matrices $\widehat{\varphi}(\alpha)$ are scalar and determined by the restriction of $\varphi$ to the \emph{central algebra}, i.e.~the $*$-subalgebra of $\O(\G)$ generated by characters. The states which we will study here are closely linked to the free product construction. Free probability theory gives tools to deal with such states. Moreover, because our states are central we only need to compute them on the central algebra, and it turns out that one of these states will be a $*$-homomorphism when restricted to characters. This considerably simplifies computations thanks to the following result :

\begin{lem}\label{lem:moments}
Let $C$ be a $*$-algebra together with a state $\varphi$ and let $A$ and $B$ be $*$-subalgebras which are $*$-free with respect to $\varphi$. If moreover $\varphi_{\vert A}$ is a $*$-homomorphism, then for any $n\in \N$ and any $a_{1}, \cdots, a_{n}\in A$ and $b_{0}, \cdots, b_{n}\in B$,
\begin{equation*}
\varphi(b_{0}a_{1}b_{1}\cdots b_{n-1}a_{n}b_{n}) = \varphi\left(\prod_{i=1}^{n}a_{i}\right)\varphi\left(\prod_{i=0}^{n} b_{i}\right).
\end{equation*}
\end{lem}

\begin{proof}
The argument is almost the same as in the proof of \cite[Thm 14.4]{nica2006lectures} but we sketch it for completeness. Using the moment-cumulant formula (see for instance \cite[Eq 11.8]{nica2006lectures}), we can write
\begin{equation*}
\varphi(b_{0}a_{1}\cdots a_{n}b_{n}) = \sum_{\pi\in NC(2n+1)}\kappa_{\pi}[b_{0}, a_{1}, b_{1}, \cdots, b_{n-1}, a_{n}, b_{n}].
\end{equation*}
Because $A$ and $B$ are $*$-free, mixed cumulants vanish (see \cite[Thm 11.16]{nica2006lectures}). Moreover, because $\varphi$ is a $*$-homomorphism when restricted to $A$, cumulants of order at least two vanish by \cite[Prop 11.7]{nica2006lectures}, so that the term corresponding to a given $\pi$ is $0$ unless its restriction to the $a_{i}$'s consists only in singletons. In that case, we have no constraint on the restriction of $\pi$ to the $b_{i}$'s, so that
\begin{equation*}
\varphi(b_{0}a_{1}b_{1}\cdots b_{n-1}a_{n}b_{n}) = \varphi\left(\prod_{i=1}^{n}a_{i}\right)\sum_{\sigma\in NC(n+1)}\kappa_{\sigma}[b_{0}, b_{1}, \cdots, b_{n}]
\end{equation*}
Using again the moment-cumulant formula, the sum in the right-hand side can be replaced by the corresponding moment and the proof is complete.
\end{proof}

For any $t > 2$, set
\begin{equation*}
q(t) = \frac{t-\sqrt{t^{2}-4}}{2}.
\end{equation*}
This is the unique real number $0 < q(t) < 1$ such that $q(t) + q(t)^{-1} = t$, and the fundamental quantity for all our estimates. Note that
\begin{equation*}
q(t)^{-1} = \frac{t+\sqrt{t^{2}-4}}{2}\geqslant \frac{t}{2}
\end{equation*}
so that $q(t) \leqslant 2/t$. Let us also define a sequence of functions $u_{n}$ by
\begin{equation*}
u_{n}(t) = \frac{q(t)^{-n-1} - q(t)^{n+1}}{q(t)^{-1} - q(t)}.
\end{equation*}
We will use repeatedly the following inequalities (see \cite[Lem 3.3]{freslon2017cutoff} for a proof) :

\begin{lem}\label{lem:encadrement}
For all $n\geqslant 1$ and $t > 2$,
\begin{equation*}
tq(t)^{-(n-1)} \leqslant u_{n}(t)\leqslant\frac{q(t)^{-n}}{1-q(t)^{2}}.
\end{equation*}
\end{lem}

Let us end this section with a simple lemma which will be useful for the study of lower bounds.

\begin{lem}\label{lem:inequalityforlowerbound}
For $a>0$ and any $N \geqslant 2a$,
\begin{equation*}
N\left(1-\frac{a}{N}\right)^{N\ln(N)/a} \geqslant \frac{e^{-a/2e}}{\sqrt{2}}.
\end{equation*}
\end{lem}

\begin{proof}
This is elementary and was used directly for instance in \cite{rosenthal1994random} to which we appealed for similar computations in \cite{freslon2017cutoff}. Let us however give a proof for the sake of completeness. Estimates for alternate series yield
\begin{equation*}
\ln(1-x)\geqslant - x - \frac{x^{2}}{2}\frac{1}{1-x} = - x - \frac{x^{2}}{2(1-x)}
\end{equation*}
so that
\begin{equation*}
\left(1-\frac{a}{N}\right)^{N/a} \geqslant \exp\left(-1-\frac{a}{2(N-a)}\right)
\end{equation*}
and
\begin{equation*}
N\left(1-\frac{a}{N}\right)^{N\ln(N)/a} \geqslant \exp\left(-\frac{a\ln(N)}{2(N-a)}\right).
\end{equation*}
The maximum of the function $x\mapsto \ln(x)/x$ is $e^{-1}$, so that
\begin{equation*}
\frac{\ln(N)}{N-a} = \frac{\ln(N-a)}{N-a} + \frac{\ln(N/(N-a))}{N-a} \leqslant e^{-1} + \frac{\ln(2)}{a}
\end{equation*}
and the result follows.
\end{proof}

\section{Free unitary quantum groups}\label{sec:unitary}

The main part of this work is concerned with free unitary quantum groups, which were introduced by S. Wang in \cite{wang1995free} and are in a sense the complex versions of the free orthogonal quantum groups studied in \cite{freslon2017cutoff}. We define them through the associated Hopf $*$-algebra.

\begin{de}
Let $\O(U_{N}^{+})$ be the universal $*$-algebra generated by $N^{2}$ elements $u_{ij}$ such that for all $1\leqslant i, j\leqslant N$,
\begin{equation*}
\displaystyle\sum_{k=1}^{N}u_{ik}u_{jk}^{*} = \delta_{ij} = \sum_{i=1}^{N} u_{ik}^{*}u_{jk}.
\end{equation*}
The formula 
\begin{equation*}
\Delta(u_{ij}) = \sum_{k=1}^{N}u_{ik}\otimes u_{kj}
\end{equation*}
extends to an algebra homomorphism $\Delta : \O(U_{N}^{+})\to \O(U_{N}^{+})\otimes \O(U_{N}^{+})$ and this can be completed into a compact quantum group structure.
\end{de}

This defines the \emph{free unitary quantum group} $U_{N}^{+}$ and the abelianization of $\O(U_{N}^{+})$ is the algebra of regular functions on the classical unitary group $U_{N}$. As mentionned above, $U_{N}^{+}$ can be seen as a "complex version" of $O_{N}^{+}$ and this provides a convenient description of its characters. Let $\T$ be the compact group of complex numbers of modulus one and consider the free product (over $\C$) $\O(O_{N}^{+})\ast\O(\T)$. It has a canonical compact quantum group structure by \cite{wang1995free}. Moreover, denoting by $v_{ij}$ the generators of $\O(O_{N}^{+})$ and by $z$ the identity function on $\T$, it was proven in \cite[Prop 7]{banica1997groupe} that the map $u_{ij}\mapsto v_{ij}z$ for all $1\leqslant i, j\leqslant N$ extends to an isomorphism of compact quantum groups. It then follows from the description of the representation theory of free products in \cite{wang1995free} that the characters of $U_{N}^{+}$ are products of characters of $O_{N}^{+}$ and powers of $z$. A more precise description was given in \cite[Prop 4.3]{vergnioux2013k}, which we reproduce here :

\begin{prop}
Let $(\chi_{n})_{n\in \N}$ denote the characters of the irreducible representations of $O_{N}^{+}$. Then, the characters of the irreducible representations of $U_{N}^{+}$ are the elements
\begin{equation*}
z^{[\epsilon_{0}]_{-}}\chi_{n_{1}}z^{\epsilon_{1}}\cdots z^{\epsilon_{p-1}}\chi_{n_{p}}z^{[\epsilon_{p}]_{+}}
\end{equation*}
where $\epsilon_{0}\in \{-1, 1\}$, $[\epsilon]_{-} = \min(\epsilon, 0)$, $[\epsilon]_{+} = \max(\epsilon, 0)$ and $\epsilon_{i+1} = (-1)^{n_{i}+1}\epsilon_{i}$. Moreover, the dimension of the corresponding representation is given by applying the counit, yielding
\begin{equation*}
\varepsilon\left(z^{[\epsilon_{0}]_{-}}\chi_{n_{1}}z^{\epsilon_{1}}\cdots z^{\epsilon_{p-1}}\chi_{n_{p}}z^{[\epsilon_{p}]_{+}}\right) = \prod_{i=1}^{p}d_{n_{i}} = \prod_{i=1}^{p}u_{n_{i}}(N).
\end{equation*}
\end{prop}

We now need to build states on $\O(U_{N}^{+})$ in order to study the associated random walk and this will be done using the free product picture above. Let us choose a state $\varphi$ on $\O(O_{N}^{+})$ and a probability measure $\nu$ on $\T$. Then, the free product state $\psi = \varphi\ast\int\dd\nu$ can be restricted to $\O(U_{N}^{+})$ and yields a state which is central as soon as $\varphi$ is. As shown in Lemma \ref{lem:moments}, free product states are easier to study when one of the states is a $*$-homomorphism on the central algebra so that we should choose $\varphi$ carefully.

As a consequence of \cite[Lem 4.2]{brannan2011approximation}, for any $t\in [0, N[$ there is a state $\varphi_{t}$ on $\O(O_{N}^{+})$ given by $\varphi_{t}(\chi_{n}) = u_{n}(t)/d_{n}$ which is a $*$-homomorphism when restricted to the central algebra $\O(O_{N}^{+})_{0}$. We will therefore study the random walk associated to the state $\varphi_{t, \nu} = \varphi_{t}\ast\int\dd\nu$ and our aim in the next two subsections is to prove that a cut-off phenomenon occurs. The proof will be split into two parts, first the upper bound and then the lower bound. Note that in view of Lemma \ref{lem:moments}, the important quantities are the numbers
\begin{equation}
\frac{\varphi_{t, \nu}(z^{[\epsilon_{0}]_{-}}\chi_{n_{1}}z^{\epsilon_{1}}\cdots z^{\epsilon_{p-1}}\chi_{n_{p}}z^{[\epsilon_{p}]_{+}})}{\varepsilon\left(z^{[\epsilon_{0}]_{-}}\chi_{n_{1}}z^{\epsilon_{1}}\cdots z^{\epsilon_{p-1}}\chi_{n_{p}}z^{[\epsilon_{p}]_{+}}\right)} = m_{\epsilon}(\nu)^{2k}\prod_{i=1}^{p}\frac{u_{n_{i}}(t)^{2k}}{u_{n_{i}}(N)^{2k-2}}
\end{equation}
where $\epsilon = [\epsilon_{0}]_{-} + \epsilon_{1} + \cdots + \epsilon_{p-1} + [\epsilon_{p}]_{+}$ and $m_{\epsilon}(\nu)$ is the $\epsilon$-th moment of $\nu$.

\subsection{Upper bound}

We start with the upper bound. Since $\nu$ is a probability measure on $\T$, its moments all have modulus less than one so that the moments of the characters can be bounded by a quantity which only depends on $t$. This suggests to look for a cut-off parameter independent from $\nu$. In the case of $O_{N}^{+}$, it was shown in \cite{freslon2017cutoff} that writing $t = N-\tau$, the cut-off parameter is exactly $N\ln(N)/\tau$. We will now see that the same formula holds for $U_{N}^{+}$.

Setting $t = N-\tau$ and $k = N\ln(N)/\tau + cN$, we want to prove that the total variation distance between $\varphi^{\ast k}_{t, \nu}$ and $h$ is less that $C_{0}e^{-ct}$ for some constant $C_{0}$. We will use estimates for $O_{N}^{+}$ obtained in \cite[Sec 3]{freslon2017cutoff} which are only valid if $N$ is large enough. In particular, in \cite[Lem 3.9]{freslon2017cutoff} it is proven that setting
\begin{equation*}
C(\tau) = \frac{2}{\tau\sqrt{5}}(2+\sqrt{2 + 9\tau^{2}}),
\end{equation*}
we have
\begin{equation}\label{eq:maininequality}
Nq(N-\tau)(1-q(N-\tau)^{2}) \geqslant e^{\tau/N}
\end{equation}
as soon as $N\geqslant \tau + C(\tau)$. With this in hand we can establish the upper bound.

\begin{prop}\label{prop:unitarycutoff}
For any $N\geqslant \tau + C(\tau)$ and any $c > \ln(2)/2\tau$,
\begin{equation*}
\left\|\varphi_{N-\tau, \nu}^{\ast N\ln(N)/\tau + cN} - h\right\|_{TV} \leqslant \frac{1}{\sqrt{2-4e^{-2c\tau}}}e^{-c\tau}.
\end{equation*}
\end{prop}

\begin{proof}
Lemma \ref{lem:upperbound} reduces the problem to the computation of a series where the sum runs over all irreducible representations of $U_{N}^{+}$. Noticing that in an irreducible character, the $\epsilon$-sequence is completely determined by $\epsilon_{0}$ as soon as the irreducible representations $n_{1}, \cdots, n_{p}$ of $O_{N}^{+}$ are fixed, the interesting quantity is :
\begin{equation*}
A_{k}(N-\tau) = \sum_{p=1}^{+\infty}\sum_{n_{1}, \cdots, n_{p}\geqslant 1}\sum_{\epsilon_{0}\in \{-1, 1\}}\vert m_{\epsilon}(\nu)\vert^{2k}\prod_{i=1}^{p}\frac{u_{n_{i}}(N-\tau)^{2k}}{u_{n_{i}}(N)^{2k-2}}.
\end{equation*}
Since $\vert z\vert = 1$ for all $z\in \T$, $\vert m_{\epsilon}(\nu)\vert \leqslant 1$ for all $\epsilon$ and the sum over $\epsilon_{0}$ can be removed and replaced by a multiplicative factor $2$. After these simplifications, the estimates of Lemma \ref{lem:encadrement} yield
\begin{align*}
A_{k}(N-\tau) & \leqslant 2\sum_{p=1}^{+\infty}\sum_{n_{1}, \cdots, n_{p}\geqslant 1}\frac{q(N)^{(2k-2)(\sum n_{i}-p)}}{q(N-\tau)^{2k\sum n_{i}}}\left(\frac{1}{N^{2k-2}(1-q(N-\tau)^{2})^{2k}}\right)^{p} \\ 
 & = 2\sum_{p=1}^{+\infty}q(N-\tau)^{-2kp}\left(\frac{1}{N^{2k-2}(1-q(N-\tau)^{2})^{2k}}\right)^{p}\sum_{n_{1}, \cdots, n_{p}\geqslant 1}\left(\frac{q(N)^{2k-2}}{q(N-\tau)^{2k}}\right)^{\sum n_{i}-p}
\end{align*}
In the last line above, we have for each $p$ a sum over $p$-tuples which only depends on $n = n_{1} + \cdots + n_{p}$. This suggests to rewrite it using integer partitions. If $\pi_{p}(n)$ denotes the number of partitions of $n$ into exactly $p$ parts, the associated generating function is
\begin{equation*}
\sum_{n=p}^{+\infty}\pi_{p}(n)x^{n} = \frac{x^{p}}{(1-x)\cdots(1-x^{p})} = \prod_{i=1}^{p}\frac{x}{1-x^{i}} \leqslant \left(\frac{x}{1-x}\right)^{p}
\end{equation*}
for all $0 < x < 1$. The proof is elementary and we refer the reader for instance to \cite{andrews1998theory}. Thus,
\begin{align*}
A_{k}(N-\tau) & \leqslant 2\sum_{p=1}^{+\infty}\left(\frac{1}{N^{2k-2}q(N-\tau)^{2k}(1-q(N-\tau)^{2})^{2k}}\right)^{p}\sum_{n=p}^{+\infty}\pi_{p}(n)\left(\frac{q(N)^{2k-2}}{q(N-\tau)^{2k}}\right)^{n-p} \\
& \leqslant 2\sum_{p=1}^{+\infty}\left(\frac{1}{N^{2k-2}q(N-\tau)^{2k}(1-q(N-\tau)^{2})^{2k}}\right)^{p}\left(\frac{1}{1-\left(\frac{q(N)^{2k-2}}{q(N-\tau)^{2k}}\right)}\right)^{p}.
\end{align*}
Recall from the proof of \cite[Lem 3.8]{freslon2017cutoff} and Equation \eqref{eq:maininequality} that under our assumption on $N$,
\begin{equation*}
\frac{q(N)^{2k-2}}{q(N-\tau)^{2k}}\leqslant e^{-2\tau c} \text{ and } \frac{1}{N^{2k}q(N-\tau)^{2k}(1-q(N-\tau)^{2})^{2k}}\leqslant \frac{e^{-2\tau c}}{N^{2}}
\end{equation*}
so that
\begin{align*}
A_{k}(\tau) & \leqslant 2\sum_{p=1}^{+\infty}\left(\frac{e^{-2\tau c}}{1-e^{-2\tau c}}\right)^{p} \\
& = \frac{2e^{-2\tau c}}{1 - 2e^{-2\tau c}}
\end{align*}
where the sum is finite because of our assumption on $c$. Applying Lemma \ref{lem:upperbound} then yields the result.
\end{proof}

The statement of Proposition \ref{prop:unitarycutoff} seems incomplete since we have to assume that $c > \ln(2)/2\tau$. This could be removed by using $N\ln(N)/\tau + \ln(2)N/2\tau$ as the cut-off parameter. However, since the second term has strictly lower order than the first one, only $N\ln(N)/\tau$ is meaningful (see for instance Proposition \ref{prop:firstordercutoff} for an illustration of this idea). Note that in \cite[Thm C.1]{porod1996cutII}, U. Porod also has to assume that $c$ is greater than some (non-explicit) constant $c_{0}$.

\subsection{Lower bound}

To complete the proof of the cut-off phenomenon, we will now prove that if $k = N\ln(N)/\tau - cN$, then the total variation distance is close to $1$. One efficient strategy to do this uses the Chebyshev inequality and therefore requires to consider a self-adjoint element. The natural candidate is $\chi = \chi_{1}z + \overline{z}\chi_{1}$. However, if $m_{1}(\nu) = 0 = m_{-1}(\nu)$ (for instance if $\nu$ is the Haar measure of $\T$) then $\varphi_{t, \nu}(\chi) = 0$ and we cannot use this to produce a lower bound. We will therefore use $\chi_{2}$ instead, which makes the computations a bit more involved but works whatever the measure $\nu$ is.

\begin{prop}\label{prop:lowerboundunitary}
Let us set
\begin{equation*}
D(\tau) = \frac{2}{\tau} + 2\tau + \sqrt{\frac{3\tau^{2}}{2} + 3}.
\end{equation*}
Then, for any $N\geqslant D(\tau)$ and any $c > 0$,
\begin{equation*}
\left\|\varphi_{N-\tau, \nu}^{\ast N\ln(N) - cN} - h\right\|_{TV}\geqslant 1 - \frac{405}{8}e^{6\tau}e^{-2c\tau}.
\end{equation*}
\end{prop}

\begin{proof}
Let us set $k_{0} = N\ln(N)/\tau$. The first step is to find a lower bound for the expectation
\begin{equation*}
\varphi_{N-\tau, \nu}^{\ast k_{0}}(\chi_{2}) = \frac{((N-\tau)^{2} - 1)^{k_{0}}}{(N^{2}-1)^{k_{0}-1}} = (N^{2}-1)\left(1-\frac{2N\tau - \tau^{2}}{N^{2}-1}\right)^{N\ln(N)/\tau}.
\end{equation*}
Recall the lower bound used in Lemma \ref{lem:inequalityforlowerbound} :
\begin{equation*}
\ln(1-x)\geqslant - x - \frac{x^{2}}{2(1-x)}
\end{equation*}
which we will now apply. For simplicity, let us study each term separately. Setting $x = (2N\tau - \tau^{2})/(N^{2}-1)$, we have
\begin{equation*}
k_{0}x = \ln(N)\frac{2N^{2} - N\tau}{N^{2} - 1} = 2\ln(N) + \ln(N)\frac{2-N\tau}{N^{2}-1}
\end{equation*}
which is less than $2\ln(N)$ as soon as $N\geqslant 2/\tau$. On the other hand,
\begin{align*}
k_{0}\frac{x^{2}}{2(1-x)} & = N\ln(N)\tau\left(\frac{2N - \tau}{N^{2}-1}\right)^{2}\frac{N^{2} - 1}{2((N-\tau)^{2} - 1)} \\
& = \ln(N)\tau\frac{(2N-\tau)^{2}}{2(N-\tau)^{2} - 2}\frac{N}{N^{2}-1} \\
& = \tau\frac{\ln(N)}{N-1}\frac{N}{N+1}\frac{(2N-\tau)^{2}}{2(N-\tau)^{2} - 2}.
\end{align*}
The first two fractions are less than one for any $N\geqslant 3$. As for the third one, it can be written as
\begin{equation*}
2 + \frac{4\tau N - 3\tau^{2} + 4}{2N^{2} - 4\tau N + 2\tau^{2} - 2}
\end{equation*}
and the fraction appearing above is less than one provided $2N^{2} - 8\tau N + 5\tau^{2} - 6\geqslant 0$, which is satisfied as soon as $N\geqslant 2\tau + \sqrt{3\tau^{2}/2 + 3}$. Summing up, as soon as $N\geqslant D(\tau)$,
\begin{equation*}
\varphi_{N-\tau, \nu}^{\ast N\ln(N)/\tau}(\chi_{2}) \geqslant (N^{2}-1)\exp(-2\ln(N) - 3\tau) = e^{-3\tau}\frac{N^{2}-1}{N^{2}}\geqslant \frac{8e^{-3\tau}}{9}.
\end{equation*}
Moreover, for any $N\geqslant (\tau+\sqrt{\tau^{2}-4})/2$, $2N^{2} - \tau N \geqslant N^{2} - 1$ so that
\begin{equation*}
\left(1 - \frac{2N\tau - \tau^{2}}{N^{2}-1}\right)^{-cN} \geqslant \exp\left(c\frac{2N^{2}\tau - N\tau^{2}}{N^{2}-1}\right)\geqslant e^{c\tau}
\end{equation*}
Gathering both inequalities and the fact that $(\tau+\sqrt{\tau^{2}-4})/2\leqslant \tau\leqslant D(\tau)$, we see that for $N\geqslant D(\tau)$,
\begin{equation*}
\varphi_{N-\tau, \nu}^{\ast N\ln(N)/\tau - cN}(\chi_{2}) = (N^{2}-1)\left(1-\frac{2N\tau - \tau^{2}}{N^{2}-1}\right)^{N\ln(N)/\tau}\left(1-\frac{2N\tau - \tau^{2}}{N^{2}-1}\right)^{-cN}\geqslant \frac{8e^{-3\tau}}{9}e^{c\tau}.
\end{equation*}
We can now apply the strategy of \cite[Prop 3.15]{freslon2017cutoff}, bounding the variance of $\chi_{2}$ by  $\|\chi_{2}\|_{\infty}^{2} = 9$. Since $h(\chi_{2}) = 0$ and $h(\chi_{2}^{2}) = 1$, this yields
\begin{equation*}
\left\|\varphi_{N-\tau, \nu}^{\ast N\ln(N)/\tau - cN} - h\right\|_{TV} \geqslant 1-\frac{810}{16}e^{6\tau}e^{-2c\tau}
\end{equation*}
from which the result follows.
\end{proof}

Gathering the upper and lower bound yields the announced cut-off phenomenon.

\begin{thm}\label{thm:cutoffunitary}
Let $\tau > 0$ and let $\nu$ be a probability measure on the circle. Then, for $N\geqslant \max(\tau + C(\tau), D(\tau))$ the random walk associated to $\varphi_{N-\tau, \nu}$ has a cut-off at $N\ln(N)/\tau$ steps.
\end{thm}

The striking fact in this theorem is of course that the probability measure $\nu$ has no impact on the threshold. Indeed, since the moments of $\nu$ are bounded by one, they cannot increase the cut-off parameter in the sense that $A_{k}(N-\tau)$ is always less than its version without $\nu$. However, they could improve the cut-off parameter by making the sum converge faster. This does not happen because the even characters of $O_{N}^{+}$ are also characters of $U_{N}^{+}$ and have the correct lower bound.

\subsection{Random reflections}

We now want to take a step further and study random walks coming from classical unitary groups. To our knowledge, the only cut-off phenomenon for unitary groups was proved by U. Porod in \cite{porod1996cutII} for random walks generated by reflections. More precisely, any reflection is unitarily conjugate to a diagonal matrix with coefficients $(z, 1, \cdots, 1)$ for some $z\in \T$. Thus, given a measure $\mu$ on the circle one can pick $z$ at random according to $\mu$ and then conjugate it by a Haar distributed random unitary matrix to produce a "random reflection" giving one step of the random walk.

\subsubsection{One conjugacy class}

In the quantum setting, we start with the random walk on one conjugacy class, whose construction is explained in \cite{freslon2017cutoff}. If $g\in U_{N}$ is a unitary matrix, the evaluation map $\ev_{g} : u_{ij}\mapsto g_{ij}$ induces a state on the central subalgebra of $\O(U_{N}^{+})$ which in turn yields a central state on the whole of $\O(U_{N}^{+})$. Note that since $\ev_{g}$ is a $*$-homomorphism, it is completely determined by its image on $\chi_{1}z$, which is $\Tr(g)$. Thus, it coincides with $\varphi_{\vert\Tr(g)\vert, \delta_{\Arg(\Tr(g))}}$ since this is also a $*$-homomorphism on the central algebra.

By Theorem \ref{thm:cutoffunitary}, the corresponding cut-off parameter therefore only depends on the modulus of the trace of $g$. Setting $z = e^{i\theta}$, we get $\vert\Tr(g)\vert^{2} = N^{2} - 2N(1-\cos(\theta)) + 2(1-\cos(\theta))$ so that
\begin{equation*}
N - \vert\Tr(g)\vert = N - \sqrt{N^{2} - 2N(1-\cos(\theta)) + 2(1-\cos(\theta))} = 1-\cos(\theta) + O\left(\frac{1}{N}\right).
\end{equation*}
This means that the leading term in $N\ln(N)/\tau$ is $N\ln(N)/(1-\cos(\theta))$ so that this should be the real cut-off parameter. Let us prove this.

\begin{prop}\label{prop:firstordercutoff}
Let $\varphi_{\theta}$ be the state coming from the uniform measure on the quantum conjugacy class of the diagonal matrix $g_{\theta}$ with coefficients $(e^{i\theta}, 1, \cdots, 1)$. Then, for $N\geqslant \max(\tau_{\theta} + C(\tau_{\theta}), D(\tau_{\theta}))$ the random walk associated to $\varphi_{\theta}$ has a cut-off at $N\ln(N)/(1-\cos(\theta))$ steps.
\end{prop}

\begin{proof}
For convenience, let us set $\tau_{\theta} = N - \vert\Tr(g_{\theta})\vert$ and $\lambda_{\theta} = 1-\cos(\theta)$ and note that
\begin{equation*}
(N-\tau_{\theta})^{2} = N^{2} - 2N\lambda_{\theta} + 2\lambda_{\theta} > (N-\lambda_{\theta})^{2}
\end{equation*}
so that $\lambda_{\theta} > \tau_{\theta}$. Moreover, using $\sqrt{1-x}\leqslant 1-x/2$ we have
\begin{equation*}
\tau_{\theta} = N\left(1 - \sqrt{1-\frac{2\lambda_{\theta}}{N} + \frac{2\lambda_{\theta}}{N^{2}}}\right) \geqslant \lambda_{\theta} - \frac{\lambda_{\theta}}{N}
\end{equation*}
so that
\begin{equation*}
d_{\theta} = \frac{1}{\tau_{\theta}} - \frac{1}{\lambda_{\theta}} \leqslant \frac{1}{\tau_{\theta}} - \frac{N-1}{N}\frac{1}{\tau_{\theta}} = \frac{1}{N\tau_{\theta}}.
\end{equation*}
In particular, $\ln(N)d_{\theta}\tau_{\theta}\leqslant \ln(4)/4$ for all $N\geqslant 4$ (note that $D(\tau)\geqslant 4$) and
\begin{align*}
\left\|\varphi_{\theta}^{\ast \frac{N\ln(N)}{\lambda_{\theta}} + cN} - h\right\|_{TV} & = \left\|\varphi_{\theta}^{\ast \frac{N\ln(N)}{\tau_{\theta}} + (c-d_{\theta}\ln(N))N} - h\right\|_{TV} \\
& \leqslant \frac{1}{\sqrt{2-4e^{-2c\tau_{\theta}}}}e^{d_{\theta}\tau_{\theta}\ln(N)}e^{-c\tau_{\theta}} \\
& \leqslant \frac{1}{\sqrt{1-2e^{-2c\tau_{\theta}}}}e^{-c\tau_{\theta}} \\
& \leqslant \frac{1}{\sqrt{1-2e^{-3c\lambda_{\theta}/2}}}e^{-3c\lambda_{\theta}/4}.
\end{align*}

As for the lower bound, by Proposition \ref{prop:lowerboundunitary},
\begin{align*}
\left\|\varphi_{\theta}^{\ast N\ln(N)/\lambda_{\theta} - cN} - h\right\|_{TV} & = \left\|\varphi_{\theta}^{\ast N\ln(N)/\tau_{\theta} - (c+d_{\theta}\ln(N))N} - h\right\|_{TV} \\
& \geqslant 1 - \frac{405}{6}e^{6\tau_{\theta}}e^{-2(c+d_{\theta}\ln(N))\tau_{\theta}} \\
& \geqslant 1 - \frac{405}{6}e^{6\tau_{\theta}}e^{-2c\tau_{\theta}}.
\end{align*}
\end{proof}

\subsubsection{Mixed reflections}

In \cite{porod1996cutII}, U. Porod considered a random walk on the classical unitary group $U(N)$ where the angle $\theta$ of the random reflection is chosen using a probability measure $\mu_{N}$ on $[0, 2\pi[$ with density $(4W_{N-1})^{-1}\vert \sin(\theta/2)\vert^{N-1}$, where $W_{N} = \int_{0}^{\pi/2}\sin(\theta)^{N}\dd\theta$ is the $N$-th Wallis integral. A natural idea to study the corresponding problem on $U_{N}^{+}$ would be to try to use our previous results to obtain a cut-off for this random walk and the results of \cite{hough2017cut} together with Proposition \ref{prop:firstordercutoff} tell us that the cut-off parameter should be $N\ln(N)/\lambda$ with
\begin{equation*}
\lambda = \int_{0}^{2\pi}\lambda_{\theta}\dd\mu_{N}(\theta) = \int_{0}^{2\pi} (1-\cos(\theta))\dd\mu_{N}(\theta) = 2\frac{W_{N+1}}{W_{N-1}} = 2\frac{N+1}{N+2} = 2 - \frac{2}{N+2}.
\end{equation*}
Again in the spirit Proposition \ref{prop:firstordercutoff}, the cut-off parameter would then be given by the leading term $N\ln(N)/2$, which is exactly the result of \cite{porod1996cutII} in the classical case. However, as already noticed in the end of Section 3 of \cite{freslon2017cutoff}, adapting the ideas of \cite{hough2017cut} is not enough since our upper bounds are only valid for $N$ larger than some quantity which goes to infinity as $\theta$ goes to $0$.

The proof of the cut-off phenomenon must therefore go through a direct computation. This is what we will do now. For clarity, we will split the computations into several lemmata and we first set some notations. Let $\varphi_{\mu_{N}}$ be the state on $\O(U_{N}^{+})$ defined for $x\in \O(U_{N}^{+})$ by
\begin{equation*}
\varphi_{\mu_{N}}(x) = \int_{0}^{2\pi}\varphi_{g_{\theta}}(x)\dd\mu_{N}(x) = \int_{0}^{2\pi}\varphi_{N-\tau_{\theta}, \delta_{\Arg(\Tr(g_{\theta}))}}(x)\dd\mu_{N}(x),
\end{equation*}
setting as before $\tau_{\theta} = N - \vert\Tr(g_{\theta})\vert$. The strategy is to replace $\tau_{\theta}$ by $\lambda_{\theta} = 1-\cos(\theta)$ to simplify the computations while keeping a control on the difference between the two quantities. The first problem in this approach is that $\varphi_{\mu_{N}}$ is not a free product state any more, hence Lemma \ref{lem:moments} does not apply. However,
\begin{align*}
\vert\varphi_{\mu_{N}}(z^{[\epsilon_{0}]_{-}}\chi_{n_{1}}z^{\epsilon_{1}}\cdots z^{\epsilon_{p-1}}\chi_{n_{p}}z^{[\epsilon_{p}]_{+}})\vert & = \left\vert\int_{0}^{2\pi}\varphi_{g_{\theta}}(z^{[\epsilon_{0}]_{-}}\chi_{n_{1}}z^{\epsilon_{1}}\cdots z^{\epsilon_{p-1}}\chi_{n_{p}}z^{[\epsilon_{p}]_{+}})\dd\mu_{N}(\theta)\right\vert \\
& \leqslant \int_{0}^{2\pi}\left\vert\varphi_{g_{\theta}}(z^{[\epsilon_{0}]_{-}}\chi_{n_{1}}z^{\epsilon_{1}}\cdots z^{\epsilon_{p-1}}\chi_{n_{p}}z^{[\epsilon_{p}]_{+}})\right\vert\dd\mu_{N}(\theta) \\
& \leqslant \int_{0}^{2\pi} \left\vert m_{\epsilon}(\nu)\prod_{i=1}^{p}u_{n_{i}}(N-\tau_{\theta})\right\vert\dd\mu_{N}(\theta) \\
& \leqslant \int_{0}^{2\pi}\prod_{i=1}^{p}u_{n_{i}}(N-\tau_{\theta})\dd\mu_{N}(\theta) \\
\end{align*}
and the same computation as in the beginning of the proof of Proposition \ref{prop:unitarycutoff} shows that it is enough to bound
\begin{equation*}
\int_{0}^{2\pi}\left(\frac{1}{N^{2k-2}q(N-\tau_{\theta})^{2k}(1-q(N-\tau_{\theta})^{2})^{2k}}\right)^{p}\sum_{n_{1}, \cdots, n_{p}\geqslant 1}\left(\frac{q(N)^{2k-2}}{q(N-\tau_{\theta})^{2k}}\right)^{\sum n_{i}-p}\dd\mu_{N}(\theta).
\end{equation*}
for each $p$ and then compute the sum. Moreover, we can replace $\sum n_{i}$ by just one integer $n$ if we multiply by a suitable number of integer partitions. We start by studying the corresponding quantity with $\tau_{\theta}$ replaced by $\lambda_{\theta}$.

\begin{lem}\label{lem:estimatepsi}
Let us set, for $p\geqslant 1$,
\begin{equation*}
B_{p, n}(k) = \int_{0}^{2\pi}\left(\frac{1}{N^{2k-2}q(N-\lambda_{\theta})^{2k}(1-q(N-\lambda_{\theta})^{2})^{2k}}\right)^{p}\left(\frac{q(N)^{2k-2}}{q(N-\lambda_{\theta})^{2k}}\right)^{n-p}\dd\mu_{N}(\theta).
\end{equation*}
Then, writing $a_{N} = N-2+2/N$, we have
\begin{equation*}
B_{p, n}(k) \leqslant \left(\frac{a_{N}^{2k}}{N^{2k-2}(1-q(N-2)^{2})^{2k}}\right)^{p}\left(a_{N}^{2k}q(N)^{2k-2}\right)^{n-p}.
\end{equation*}
\end{lem}

\begin{proof}
Using $q(N-\lambda_{\theta})\leqslant q(N-2)$, we can bound $(1-q(N-\lambda_{\theta}))^{-1}$ so that we only have one term left in the integral which depends on $\theta$, namely (noticing that $q(t)^{-1}\leqslant t$)
\begin{equation*}
\int_{0}^{2\pi} q(N-\lambda_{\theta})^{-2kn}\dd\mu_{N}(\theta) \leqslant \int_{0}^{2\pi} (N-\lambda_{\theta})^{2kn}\dd\mu_{N}(\theta).
\end{equation*}
Let us fix an integer $\alpha > 0$. Since $1-\cos(\theta) = 2\sin(\theta/2)^{2}$, we see that for any $\ell\leqslant \alpha$,
\begin{align*}
\int_{0}^{2\pi} \lambda_{\theta}^{\ell}\dd\mu_{N}(\theta) & = \frac{1}{4W_{N-1}}\int_{0}^{2\pi} 2^{\ell}\left\vert\sin\left(\frac{\theta}{2}\right)^{2\ell+N-1}\right\vert\dd\theta \\
& = 2^{\ell}\frac{W_{N-1+2\ell}}{W_{N-1}} \\
& = 2^{\ell}\prod_{s=1}^{\ell}\frac{N+2s}{N+2s+1} \\
& \leqslant 2^{\ell}\left(\frac{N+2\alpha}{N+2\alpha+1}\right)^{\ell}
\end{align*}
so that
\begin{align*}
\int_{0}^{2\pi} (N-\lambda_{\theta})^{\alpha}\dd\mu_{N}(\theta) & \leqslant\sum_{\ell=0}^{\alpha}\binom{\alpha}{\ell}N^{\alpha-\ell}(-1)^{\ell}2^{\ell}\left(\frac{N+2\alpha}{N+2\alpha+1}\right)^{\ell} \\
& = \left(N-2\frac{N+2\alpha}{N+2\alpha+1}\right)^{\alpha} \\
& = \left(N-2+\frac{2}{N+2\alpha+1}\right)^{\alpha} \\
& \leqslant \left(N-2+\frac{2}{N}\right)^{\alpha}.
\end{align*}
Applying this for $\alpha = 2kn$ then yields the announced estimate.
\end{proof}

If we had $q(N-2)^{-1}$ instead of $a_{N}$ in the bound of Lemma \ref{lem:estimatepsi}, then we could use the estimates from \cite[Sec 3]{freslon2017cutoff} to finish the computation. It is therefore natural to try to show that $a_{N}$ is close enough to $q(N-2)^{-1}$ so that the aformentioned estimates will still be valid.

\begin{lem}\label{lem:comparisonanqn}
For any $N\geqslant 4$,
\begin{equation*}
a_{N}q(N-2)\leqslant 1 + \frac{8}{(N-2)^{2}}.
\end{equation*}
\end{lem}

\begin{proof}
It follows from the definition that $tq(t) = 1 + q(t)^{2}$, thus
\begin{equation*}
a_{N}q(N-2) = (N-2)q(N-2) + \frac{2}{N}q(N-2) = 1 + q(N-2)^{2} + \frac{2}{N}q(N-2).
\end{equation*}
Using $q(t) \leqslant 2/t$ and $1/N\leqslant 1/(N-2)$ we get the desired inequality.
\end{proof}

As explained in the beginning of this subsection, we must also control the gap between $\lambda_{\theta}$ and $\tau_{\theta}$ and this is the content of the next lemma. We already compared their inverses in the proof of Proposition \ref{prop:firstordercutoff} but we now need to compare the values of the polynomials $u_{n}$ at $N-\lambda_{\theta}$ and $N-\tau_{\theta}$.

\begin{lem}\label{lem:comparisonpsiphi}
For any $N\geqslant 6$ and any $\theta\in [0, 2\pi[$,
\begin{equation*}
\frac{u_{n}(N-\tau_{\theta})}{u_{n}(N-\lambda_{\theta})}\leqslant e^{4n/(N-2)^{2}}.
\end{equation*}
\end{lem}

\begin{proof}
The estimates of Lemma \ref{lem:encadrement} yield
\begin{align*}
\frac{u_{n}(N-\tau_{\theta})}{u_{n}(N-\lambda_{\theta})} & \leqslant \frac{1}{(N-\lambda_{\theta})(1-q(N-\tau_{\theta})^{2})}\frac{q(N-\lambda_{\theta})^{n-1}}{q(N-\tau_{\theta})^{n}} \\
& = \frac{1}{(N-\lambda_{\theta})q(N-\tau_{\theta})(1-q(N-\tau_{\theta})^{2})}\left(\frac{q(N-\lambda_{\theta})}{q(N-\tau_{\theta})}\right)^{n-1}.
\end{align*}
Using the fact that $q(t) = (1+q(t)^{2})/t$ and the inequality $q(t)\leqslant 2/t$, we get
\begin{align*}
\frac{1}{(N-\lambda_{\theta})q(N-\tau_{\theta})(1-q(N-\tau_{\theta})^{2})} & = \frac{N-\tau_{\theta}}{N-\lambda_{\theta}}\frac{1}{1-q(N-\tau_{\theta})^{4}} \\
& = \left(1 + \frac{\lambda_{\theta} - \tau_{\theta}}{N-\lambda_{\theta}}\right)\left(1 + \frac{q(N-\tau_{\theta})^{4}}{1-q(N-\tau_{\theta})^{4}}\right) \\
& \leqslant \left(1 + \frac{\lambda_{\theta} - \tau_{\theta}}{N-\lambda_{\theta}}\right)\left(1 + \frac{16}{(N-\tau_{\theta})^{4}-16}\right).
\end{align*}
As for the other term, we need a finer estimate on $q(t)$. Recall from the proof of \cite[Lem 3.10]{freslon2017cutoff} that if $(a_{n})_{n\in \N}$ denotes the sequence of coefficients of the power series expansion of the square root function at $1$, then
\begin{equation*}
q(t) = \frac{1}{t} + \sum_{n=2}^{+\infty}a_{n}\left(\frac{2}{t}\right)^{2n-1}.
\end{equation*}
Since $a_{n}\leqslant 1/8$ for all $n\geqslant 2$, $1/t\leqslant q(t) \leqslant 1/t + 1/(t^{3}-4t)$ hence
\begin{align*}
\frac{q(N-\lambda_{\theta})}{q(N-\tau_{\theta})} & \leqslant \frac{N-\tau_{\theta}}{N-\lambda_{\theta}}\left(1 + \frac{1}{(N-\lambda_{\theta})^{2} - 4}\right) \\
& \leqslant \left( 1 + \frac{\lambda_{\theta}-\tau_{\theta}}{N-\lambda_{\theta}}\right)\left(1 + \frac{2}{(N-\lambda_{\theta})^{2}}\right) \\
\end{align*}
where we used the fact that for $N\geqslant 5$, $(N-\lambda_{\theta})^{2}\geqslant 8$ and $1/(x-4)\leqslant 2/x$ as soon as $x\geqslant 8$. It follows that
\begin{align*}
\frac{u_{n}(N-\tau_{\theta})}{u_{n}(N-\lambda_{\theta})} & \leqslant \exp\left(n\frac{\lambda_{\theta} - \tau_{\theta}}{N-\lambda_{\theta}} + (n-1)\frac{2}{(N-\lambda_{\theta})^{2}} + \frac{16}{(N-\tau_{\theta})^{4}-16}\right) \\
& \leqslant \exp\left(n\frac{\lambda_{\theta} - \tau_{\theta}}{N-\lambda_{\theta}} + (n-1)\frac{2}{(N-2)^{2}} + \frac{16}{(N-2)^{4}-16}\right) \\
& = \exp\left(n\frac{\lambda_{\theta} - \tau_{\theta}}{N-\lambda_{\theta}} + n\frac{2}{(N-2)^{2}}\right)\exp\left(-\frac{2}{(N-2)^{2}} + \frac{16}{(N-2)^{4}-16}\right)
\end{align*}
The second exponential will be less than $1$ as soon as $8(N-2)^{2} \leqslant (N-2)^{4} - 16$, which is satisfied if $N\geqslant 2 + 2\sqrt{1 + \sqrt{2}}$ and this number is less than $6$. Moreover, we have seen in the proof of Lemma \ref{lem:estimatepsi} that $\tau_{\theta}\geqslant (N-1)\lambda_{\theta}/N$ so that
\begin{equation*}
\frac{\lambda_{\theta} - \tau_{\theta}}{N-\lambda_{\theta}} \leqslant \frac{\lambda_{\theta}}{N(N-\lambda_{\theta})} \leqslant \frac{2}{N(N-2)}\leqslant \frac{2}{(N-2)^{2}}
\end{equation*}
and the result follows.
\end{proof}

We are now ready to prove the cut-off for the random walk associated to $\varphi_{\mu_{N}}$. Let us set
\begin{equation*}
A_{k}(\mu_{N}) = \sum_{p=1}^{+\infty}\sum_{n_{1}, \cdots, n_{p}\geqslant 1}\sum_{\epsilon_{0}\in \{-1, 1\}}\left\vert\int_{0}^{2\pi} m_{\epsilon}(\delta_{\Arg(g_{\theta})})\prod_{i=1}^{p}\frac{u_{n_{i}}(N-\tau_{\theta})^{2k}}{u_{n_{i}}(N)^{2k-2}}\dd\mu_{N}(\theta)\right\vert^{2k}.
\end{equation*}

\begin{thm}
For $N\geqslant 12$, the random walk associated to the state $\varphi_{\mu_{N}}$ has a cut-off at $N\ln(N)/2$ steps.
\end{thm}

\begin{proof}
We start with the upper bound and we will prove that for any $N\geqslant 12$ and any $c > 4 + \ln(2)$,
\begin{equation*}
\|\varphi_{\mu_{N}}^{\ast N\ln(N)/2 + cN} - h\|_{TV}\leqslant \frac{e^{2}}{\sqrt{2-4e^{4-c}}}e^{-c/2}.
\end{equation*}
Setting $b_{N} = e^{4/(N-2)^{2}}$, we have by Lemma \ref{lem:comparisonpsiphi}
\begin{equation*}
\int_{0}^{2\pi}\prod_{i=1}^{p}\frac{u_{n_{i}}(N-\tau_{\theta})^{2k}}{u_{n_{i}}(N)^{2k-2}}\leqslant b_{N}^{2k\sum n_{i}}\int_{0}^{2\pi}\prod_{i=1}^{p}\frac{u_{n_{i}}(N-\lambda_{\theta})^{2k}}{u_{n_{i}}(N)^{2k-2}}.
\end{equation*}
Thus, Lemma \ref{lem:estimatepsi} implies through the same computations as in Proposition \ref{prop:unitarycutoff} that
\begin{align*}
A_{k}(\mu_{N}) & \leqslant 2\sum_{p=1}^{+\infty}\sum_{n=p}^{+\infty}\pi_{p}(n)b_{N}^{2kn}B_{p, n}(k) \\
& \leqslant 2\sum_{p=1}^{+\infty}\left(\frac{(a_{N}b_{N})^{2k}}{N^{2k-2}(1-q(N-2)^{2})^{2k}}\right)^{p}\sum_{n=p}^{+\infty}\pi_{p}(n)(q(N)^{2k-2}(a_{N}b_{N})^{2k})^{n-p} \\
& \leqslant 2\sum_{p=0}^{+\infty}\left(\frac{(a_{N}b_{N})^{2k}}{N^{2k-2}(1-q(N-2)^{2})^{2k}}\right)^{p}\left(\frac{1}{1-q(N)^{2k-2}(a_{N}b_{N})^{2k}}\right)^{p}.
\end{align*}
By Lemma \ref{lem:comparisonanqn}, Lemma \ref{lem:comparisonpsiphi} and the estimates of \cite[Lem 3.8 and Lem 3.10]{freslon2017cutoff} (which are valid since $N\geqslant 2 + C(2)$),
\begin{align*}
q(N)^{2k-2}(a_{N}b_{N})^{2k} & = \frac{q(N)^{2k-2}}{q(N-2)^{2k}}(q(N-2)a_{N})^{2k}b_{N}^{2k} \\
& \leqslant e^{-4c}\exp\left(\left(\frac{N\ln(N)}{2} + cN\right)\left(\frac{16}{(N-2)^{2}} + \frac{8}{(N-2)^{2}}\right)\right) \\
& \leqslant e^{-4c}e^{12N\ln(N)/(N-2)^{2}}e^{24cN/(N-2)^{2}} \\
& \leqslant e^{4-c}.
\end{align*}
Here we have used the following two elementary facts : for $N\geqslant 12$,
\begin{itemize}
\item $24N/(N-2)^{2}\leqslant 3$,
\item $12N\ln(N)/(N-2)^{2}\leqslant 4$.
\end{itemize}
Similarly,
\begin{equation*}
\frac{(a_{N}b_{N})^{2k}}{N^{2k-2}(1-q(N-2)^{2})^{2k}}\leqslant e^{4-c}
\end{equation*}
so that as soon as $c > 4$, everything is summable and
\begin{equation*}
A_{k}(\mu_{N})\leqslant 2\sum_{p=1}^{+\infty}\left(\frac{e^{4-c}}{1-e^{4-c}}\right)^{p}.
\end{equation*}
For $c > 4 + \ln(2)$, this sum converges and we conclude as in Proposition \ref{prop:unitarycutoff} that
\begin{equation*}
\|\varphi_{\mu_{N}}^{\ast N\ln(N)/2 + cN} - h\|_{TV}\leqslant \frac{e^{2}}{\sqrt{2-4e^{4-c}}}e^{-c/2}.
\end{equation*}

For the lower bound, we will again consider the self-adjoint element $\chi = \chi_{1}z + \overline{z}\chi_{1}$. We have
\begin{equation*}
\int_{0}^{2\pi}\cos(\theta)\dd\mu_{N}(\theta) = \int_{0}^{2\pi}1-\sin(\theta/2)^{2}\dd\mu_{N}(\theta) = 1 - \frac{W_{N+1}}{W_{N-1}}
\end{equation*}
while by parity the integral of $\sin(\theta)$ with respect to $\mu_{N}$ vanishes. It follows that
\begin{equation*}
\varphi_{\mu_{N}}(\chi) = 2\left(N-1+1-\frac{W_{N+1}}{W_{N-1}}\right) = N-2+\frac{2}{N+2},
\end{equation*}
hence by Lemma \ref{lem:inequalityforlowerbound} (with $a=2$),
\begin{align*}
\varphi_{\mu_{N}}^{\ast k}(\chi) & = 2N\left(1 - \frac{2}{N} + \frac{2}{N(N-2)}\right)^{k} \\
& \geqslant 2N\left(1 - \frac{2}{N}\right)^{N\ln(N)/2}\left(1 - \frac{2}{N}\right)^{-cN} \\
& \geqslant \frac{e^{-1/e}}{\sqrt{2}}e^{2c}
\end{align*}
Moreover, $\var_{\varphi_{\mu_{N}^{\ast k}}}(\chi)\leqslant \|\chi\|_{\infty}^{2}\leqslant 4$ and $h(\chi) = 0$ while $h(\chi^{2}) = 2$ so that the general method of \cite[Prop 3.15]{freslon2017cutoff} yields
\begin{equation*}
\|\varphi_{\mu_{N}^{\ast k}} - h\|_{TV} \geqslant 1 - 8e^{2/e}e^{-2c}(4 + 2) \geqslant 1 - 101e^{-2c}. 
\end{equation*}
\end{proof}

\section{Free wreath products}\label{sec:reflection}

The quantum reflection groups $H_{N}^{s+}$ were introduced by T. Banica in R. Vergnioux in \cite{banica2009fusion} as quantum versions of the complex reflection groups $H_{N}^{s}$ of matrices with one $s$-th root of unity in each row and column and all other coefficients equal to $0$. It is well-known that these can also be seen as wreath products via the isomorphism $H_{N}^{s}\simeq \Z_{s}\wr S_{N}^{+}$. In the quantum case, this wreath product decomposition has an analogue involving the notion of \emph{free wreath product} introduced by J. Bichon in \cite{bichon2004free} : to any compact quantum group $\G$ and integer $N$ it associates a new compact quantum group $\G\wr_{\ast}S_{N}^{+}$. A particular case is when $\G = \widehat{\Gamma}$ is the dual of a discrete group $\Gamma$, and when $\Gamma = \Z_{s}$ we recover the quantum reflection groups $H_{N}^{s+}$. Our aim is to prove a cut-off result for central random walks on quantum reflection groups and more generally on free wreath products of duals of finite groups by $S_{N}^{+}$.

For this purpose, we need a description of the representation theory of $\widehat{\Gamma}\wr_{\ast}S_{N}^{+}$ as well as a method to construct explicit central states. Both can be obtained thanks to the notion of \emph{monoidal equivalence} from \cite{bichon2006ergodic}, which is a useful technical tool from quantum group theory. Instead of introducing it, we will simply give the results needed in the sequel :
\begin{itemize}
\item It follows from \cite[Prop 6.3]{freslon2012examples} that if two compact quantum groups are monoidally equivalent, then there is a bijection between their sets of equivalence classes of irreducible representations which induces a bijection between their central states,
\item By \cite[Thm 5.11]{lemeux2014free}, $\widehat{\Gamma}\wr_{\ast} S_{N}^{+}$ is monoidally equivalent to the quantum subgroup of $SU_{q(\sqrt{N})}(2)\ast \widehat{\Gamma}$ generated by the coefficients of the irreducible representations of the form $u^{1}\gamma u^{1}$ for all $\gamma\in \Gamma$. Here, $SU_{q}(2)$ is S.L. Woronowicz' quantum $SU(2)$ group introduced in \cite{woronowicz1987twisted}.
\end{itemize}
To complete the description, we have to explain what is the representation theory of $SU_{q}(2)$, but this is the same as $O_{N}^{+}$ except that the dimension of the $n$-th irreducible representation is $u_{n}(q+q^{-1}) = (q^{-n-1} - q^{n+1})/(q^{-1} - q)$. On can then show (see for instance the beginning of Section $3$ of \cite{freslon2017torsion}) that the characters are all words of the form
\begin{equation*}
\chi_{2n_{0}+1}\gamma_{1}\chi_{2n_{1}+2}\gamma_{2}\cdots \chi_{2n_{p-1}+2}\gamma_{p}\chi_{2n_{p}+1}.
\end{equation*}

Now, given a state $\varphi_{t}$ on $\O(SU_{q}(2))$ and any state $\psi$ on $\O(\widehat{\Gamma}) = \C[\Gamma]$, we can form their free product, restrict it to the subgroup indicated above and then transfer it through monoidal equivalence. Note that because $SU_{q}(2)$ is not of Kac type, it is not clear that $\varphi_{t} : \chi_{n}\mapsto u_{n}(t)$ extends to a central state (there is no Haar state preserving conditional expectation onto the central algebra). This is nevertheless true by \cite[Prop 9]{freslon2013ccap}, so let $\psi$ be a state on $\O(\widehat{\Gamma})$ and let $\varphi_{t, \psi}$ be the corresponding state on $\O(\widehat{\Gamma}\wr_{\ast}S_{N}^{+})$. The right-hand side of the inequality in Lemma \ref{lem:upperbound} can be written as
\begin{equation*}
A_{k}(t) = \sum_{p=0}^{+\infty}\sum_{n_{0}, \cdots, n_{p}}\sum_{\gamma_{1}, \cdots, \gamma_{p}}\vert\psi(\gamma_{1}\cdots\gamma_{p})\vert\frac{u_{2n_{0}+1}(t)^{2k}u_{2n_{p}+1}(t)^{2k}}{u_{2n_{0}+1}(\sqrt{N})^{2k-2}u_{2n_{p}+1}(\sqrt{N})^{2k-2}}\prod_{i=1}^{p-1}\frac{u_{2n_{i}+2}(t)^{2k}}{u_{2n_{i}+2}(\sqrt{N})^{2k-2}}.
\end{equation*}
The expression above is slightly ambiguous as far as the term for $p=0$ is concerned. We write it this way for concision, but when computing we will distinguish in the end the case $p = 0$ from the other cases. As for free unitary quantum groups in Section \ref{sec:unitary}, we will see that the cut-off parameter does not depend on the state $\psi$. Note that because $\psi$ has norm one and any element of $\Gamma$ is a unitary, we have the inequality
\begin{equation}\label{eq:grouppart}
\sum_{\gamma_{1}, \cdots, \gamma_{p}\in \Gamma}\vert\psi(\gamma_{1}\cdots\gamma_{p})\vert\leqslant \vert\Gamma\vert^{p}.
\end{equation}

As usual, we will consider a parameter $t$ of the form $N-\tau$. This means that we need estimates for expressions involving powers of the function $q$ evaluated at $\sqrt{N}$ and $\sqrt{N-\tau}$. For $\tau = 2$, such estimates where obtained in the proof of \cite[Thm 4.4]{freslon2017cutoff}. We will now expand the argument to other values of $\tau$. However, and quite surprisingly, we will see that we need to assume that $\tau$ is larger that $7/4$ to get the required estimate. We are therefore not able to prove as general a result as \cite[Lem 3.10]{freslon2017cutoff}.

\begin{prop}\label{prop:estimatesquantumreflectiongroup}
There exist a degree four polynomial $Q\in \R[X]$ such that for any $\tau > 7/4$ and $N\geqslant Q(\tau)/(4\tau - 7)$,
\begin{equation*}
\frac{q(\sqrt{N})}{\sqrt{N}q(\sqrt{N-\tau})^{2}(1-q(\sqrt{N-\tau})^{2})}\leqslant e^{-\tau/N}.
\end{equation*}
\end{prop}

\begin{proof}
Proceeding as in the proof of \cite[Thm 4.4]{freslon2017cutoff} and using the fact that $e^{\tau/(t+\tau)}\leqslant 1 + \tau/t$, we know that for $t = N - \tau \geqslant 12$ the inequality in the statement is implied by the inequality
\begin{equation*}
\sqrt{t+\tau}\frac{\sqrt{t+\tau} + \sqrt{t+\tau - 4}}{2}\left(\frac{1}{t} + \frac{1}{t^{2}} + \frac{1}{2t^{3}}\right) \geqslant 1 + \frac{\tau}{t}.
\end{equation*}
Developing the left-hand side and letting aside the terms with $t^{3}$ at the denominator yields the stronger inequality
\begin{equation*}
1 + \frac{\tau}{t} \leqslant \frac{1}{2} + \frac{1}{2t} + \frac{1}{4t^{2}} + \frac{\tau}{2t} + \frac{\tau}{2t^{2}} + \frac{\sqrt{(t+\tau)(t+\tau-4)}}{2t} + \frac{\sqrt{(t+\tau)(t+\tau-4)}}{2t^{2}}.
\end{equation*}
which, after multiplying by $2t^{2}$, yields
\begin{equation}\label{eq:inequality}
t^{2} + (\tau-1)t \leqslant \frac{1}{2} + \tau + (t+1)\sqrt{(t+\tau)(t+\tau-4)}.
\end{equation}
Let us set $f(t) = \sqrt{(t+\tau)(t+\tau-4)}$. Using the lower bound $\sqrt{1+x}\geqslant 1+x/2-x^{2}/8$ yields
\begin{align*}
f(t) & = t\sqrt{1 + \frac{2\tau - 4}{t} + \frac{\tau^{2}-4\tau}{t^{2}}} \\
& \geqslant t\left(1 + \frac{\tau-2}{t} + \frac{\tau^{2} - 4\tau}{2t^{2}} - \frac{\tau^{2}-4\tau+4}{2t^{2}}-\frac{(\tau-2)(\tau^{2}-4\tau)}{2t^{3}} - \frac{(\tau^{2}-4\tau)^{2}}{8t^{4}}\right) \\ 
& = t + (\tau-2) - \frac{2}{t} - \frac{(\tau - 2)(t^{2} - 4\tau)}{2t^{2}} - \frac{(\tau^{2} - 4\tau)^{2}}{8t^{3}}
\end{align*}
so that
\begin{equation*}
(t+1)f(t) \geqslant t^{2} + (\tau-1)t + \tau - 4 - \frac{4 + \tau^{3} - 6\tau^{2} + 8\tau}{2t} - \frac{(\tau^{2}-4\tau)(\tau^{2}-8)}{8t^{2}} - \frac{(\tau^{2} - 4\tau)^{2}}{8t^{3}}.
\end{equation*}
Since $t\geqslant 14$, it is enough to have
\begin{equation}\label{eq:Qcondition}
\frac{4\times 14^{2}(4 + \tau^{3} - 6\tau^{2} + 8\tau) + 14(\tau^{2}-4\tau)(\tau^{2}-8) + (\tau^{2} - 4\tau)^{2}}{8\times 14^{2}\times t}\leqslant 2\tau - \frac{7}{2}
\end{equation}
and the proof is complete, noticing that the above condition and $N\geqslant \tau + 12$ can be expressed with one polynomial $Q(\tau)$.
\end{proof}

It is easy to extract from the above proof an explicit polynomial $Q$ which works. We did not do it in order to lighten the proof, but let us give it as a separate statement :

\begin{lem}\label{lem:estimateQ}
In the statement of Proposition \ref{prop:estimatesquantumreflectiongroup}, one can take the polynomial
\begin{equation*}
Q(\tau) = \frac{\tau^{4}}{28} + 2\tau^{3} - 8\tau^{2} + 59\tau - 76.
\end{equation*}
\end{lem}

\begin{proof}
Let us denote by $P(\tau)$ the numerator of the left-hand side of Equation \eqref{eq:Qcondition}. To take into account the extra condition $N\geqslant \tau + 12$, we have to consider the polynomial $R(\tau) = P(\tau) + (4\tau - 7)(\tau  + 12)$. Developping and simplifying then yields
\begin{equation*}
R(\tau) = \frac{\tau^{4} + 52\tau^{3} - 232\tau^{2} + 1628\tau - 2128}{28}
\end{equation*}
which is less than the polynomial in the statement.
\end{proof}

Note that using this and micking the proof of \cite[Thm 4.4]{freslon2017cutoff}, one obtains a proof of the cut-off for the uniform quantum random walk on $m$-cycles for arbitrary $m$ mentioned in the end of \cite[Sec 4.1]{freslon2017cutoff}. We are now ready for the proof of the cut-off phenomenon for free wreath products.

\begin{thm}\label{thm:cutoffwreath}
Let $\tau > 7/4$ and $N\geqslant Q(\tau)/(4\tau - 7)$. Then, the random walk associated to $\varphi_{N-\tau, \psi}$ has a cut-off at $N\ln(N)/\tau$ steps.
\end{thm}

\begin{proof}
As usual, we start with the upper bound and we will prove that for $k = N\ln(N)/\tau$ and any $c > \ln(1+\sqrt{\vert\Gamma\vert})/2\tau$,
\begin{equation*}
\left\|\varphi_{N-\tau, \psi}^{\ast N\ln(N)/\tau + cN} - h\right\|_{TV}\leqslant \frac{e^{-\tau c}}{\sqrt{1 - 1e^{-2\tau c}}}\sqrt{1+\sqrt{\vert\Gamma\vert}}.
\end{equation*}
This requires bounding $A_{k}(\sqrt{N-\tau})$ and we will first consider the sum over $(n_{0}, \cdots, n_{p})$. By Lemma \ref{lem:encadrement},
\begin{align*}
& \sum_{n_{0}, \cdots, n_{p}\geqslant 1} \frac{u_{2n_{0}+1}(\sqrt{N-\tau})^{2k}}{u_{2n_{0}+1}(\sqrt{N})^{2k-2}}\frac{u_{2n_{p}+1}(\sqrt{N-\tau})^{2k}}{u_{2n_{p}+1}(\sqrt{N})^{2k-2}}\prod_{i=1}^{p-1}\frac{u_{2n_{i}+2}(\sqrt{N-\tau})^{2k}}{u_{2n_{i}+2}(\sqrt{N})^{2k-2}} \\
\leqslant & \sum_{n_{0}, \cdots, n_{p}\geqslant 1} \frac{q(\sqrt{N})^{(2k-2)(p-1+2\sum n_{i})}}{q(\sqrt{N-\tau})^{2k(2p + 2\sum n_{i})}}\left(\frac{1}{\sqrt{N}^{2k-2}(1-q(\sqrt{N-\tau})^{2})^{2k}}\right)^{p+1} \\
= & \frac{q(\sqrt{N})^{(2k-2)(p-1)}}{q(\sqrt{N-\tau})^{4kp}}\left(\frac{1}{\sqrt{N}^{2k-2}(1-q(\sqrt{N-\tau})^{2})^{2k}}\right)^{p+1}\sum_{n=p+1}^{+\infty}\pi_{p+1}(n)\left(\frac{q(\sqrt{N})^{4k-4}}{q(\sqrt{N-\tau})^{4k}}\right)^{n} \\
\leqslant & \frac{q(\sqrt{N})^{(2k-2)(p-1)}}{q(\sqrt{N-\tau})^{4kp}}\left(\frac{1}{\sqrt{N}^{2k-2}(1-q(\sqrt{N-\tau})^{2})^{2k}}\right)^{p+1}\left(\frac{q(\sqrt{N})^{4k-4}}{q(\sqrt{N-\tau})^{4k}}\right)^{p+1}\\
& \times\left(\frac{1}{1 - \frac{q(\sqrt{N})^{4k-4}}{q(\sqrt{N-\tau})^{4k}}}\right)^{p+1} \\
= & \left(\frac{q(\sqrt{N})^{4k-4}}{q(\sqrt{N-\tau})^{4k}}\right)^{p}\left(\frac{q(\sqrt{N})^{2k-2}}{\sqrt{N}^{2k-2}q(\sqrt{N-\tau})^{4k}(1-q(\sqrt{N-\tau})^{2})^{2k}}\right)^{p+1}\left(\frac{1}{1 - \frac{q(\sqrt{N})^{4k-4}}{q(\sqrt{N-\tau})^{4k}}}\right)^{p+1}.
\end{align*}
It follows from the proof of \cite[Lem 3.8]{freslon2017cutoff} that
\begin{equation*}
\frac{q(\sqrt{N})^{4k-4}}{q(\sqrt{N-\tau})^{4k}}\leqslant e^{-2\tau c}
\end{equation*}
and under our assumption on $N$, Proposition \ref{prop:estimatesquantumreflectiongroup} yields (through a computation similar to that of \cite[Thm 4.4]{freslon2017cutoff})
\begin{equation*}
\frac{q(\sqrt{N})^{2k-2}}{\sqrt{N}^{2k-2}q(\sqrt{N-\tau})^{4k}(1-q(\sqrt{N-\tau})^{2})^{2k}}\leqslant e^{-2\tau c}.
\end{equation*}
Combining both estimates with Equation \eqref{eq:grouppart} and separating the term for $p=0$, we get
\begin{align*}
A_{k}(\sqrt{N-\tau}) & \leqslant \frac{e^{-2\tau c}}{1-e^{-2\tau c}} + \sum_{p=1}^{+\infty}\vert\Gamma\vert^{p} \left(e^{-2\tau c}\right)^{p}\left(\frac{e^{-2\tau c}}{1-e^{-2\tau c}}\right)^{p+1}.
\end{align*}
The series converges for $c\geqslant \frac{1}{2\tau}\ln(1+\sqrt{\vert\Gamma\vert})\geqslant\frac{1}{2\tau}\ln\left(\frac{1 + \sqrt{4\vert\Gamma\vert+1}}{2}\right)$ and
\begin{align*}
A_{k}(\sqrt{N-\tau}) & \leqslant \frac{e^{-2\tau c}}{1-e^{-2\tau c}} + \frac{e^{-2\tau c}}{1-e^{-2\tau c}}\frac{\frac{\vert\Gamma\vert e^{-4\tau c}}{1-e^{-2\tau c}}}{1 - \frac{\vert\Gamma\vert e^{-4\tau c}}{1-e^{-2\tau c}}} \\
& =\frac{e^{-2\tau c}}{1-e^{-2\tau c}}\left(1 + \frac{\vert\Gamma\vert e^{-4\tau c}}{1 - e^{-2\tau c} - \vert\Gamma\vert e^{-4\tau c}}\right).
\end{align*}
The assumption on $c$ implies that
\begin{equation*}
\frac{\vert\Gamma\vert e^{-4\tau c}}{1 - e^{-2\tau c} - \vert\Gamma\vert e^{-4\tau c}}\leqslant \frac{\vert\Gamma\vert}{(1 + \sqrt{\vert\Gamma\vert})^{2} - (1+\sqrt{\vert\Gamma\vert}) - \vert\Gamma\vert} = \sqrt{\vert\Gamma\vert}
\end{equation*}
and the claim follows.

As for the lower bound, we simply have to consider the character $\chi_{2}$. Its expectation and variance under the Haar state $h$ are respectively equal to $0$ and $1$. 
Using the same inequality as in Lemma \ref{lem:inequalityforlowerbound}, we see that for $N\geqslant 2+\tau$
\begin{align*}
(N-1)\left(1 - \frac{\tau}{N-1}\right)^{N\ln(N)/\tau} & \geqslant \frac{N-1}{N}e^{-\frac{\ln(N)}{N-1}(1 + N\tau/2(N-1-\tau))} \\
& \geqslant \frac{3}{4}e^{-e^{-1} - e^{-1}\tau(\tau+2)/2} \\
& \geqslant \frac{1}{2}e^{-\tau(\tau+2)/5}.
\end{align*}
so that
\begin{align*}
\varphi_{N-\tau, \psi}^{\ast k}(\chi_{2}) &  = (N-1)\left(\frac{N-\tau-1}{N-1}\right)^{k} \\
& \geqslant (N-1)\left(\frac{N-\tau-1}{N-1}\right)^{N\ln(N)/\tau}e^{cN\frac{\tau}{N-1}} \\
& \geqslant \frac{1}{2}e^{-\tau(\tau+2)/5}e^{2\tau c}.
\end{align*}
Moreover, $\var_{\varphi_{N-\tau, \psi}^{\ast k}}(\chi_{2})\leqslant \|\chi_{2}\|_{\infty}^{2}\leqslant 9$ so that (using $\tau > 7/4$)
\begin{equation*}
\|\varphi_{N-\tau, \psi}^{\ast k} - h\|_{TV} \geqslant 1 - 40e^{-2\tau(\tau+2)/5}e^{4\tau c}.
\end{equation*}
\end{proof}

\bibliographystyle{amsplain}
\bibliography{../../../quantum}

\end{document}